\documentclass[reqno,centertags,12pt]{amsart}

\usepackage{amssymb,amsmath,amsthm}

\makeatletter \def\@strippedMR{} \def\@scanforMR#1#2#3\endscan{%
  \ifx#1M\ifx#2R\def\@strippedMR{#3}%
  \else\def\@strippedMR{#1#2#3}%
  \fi\fi} \renewcommand\MR[1]{\relax \ifhmode\unskip\spacefactor3000
  \space\fi \@scanforMR#1\endscan
  MR\MRhref{\@strippedMR}{\@strippedMR}} \makeatother

\newcommand{\R}{\mathbb{R}} \newcommand{\Z}{\mathbb{Z}}
 
\newcommand{\N}{\mathbb{N}}

\theoremstyle{plain} \newtheorem{thm}{Theorem}[section]
\newtheorem{lem}[thm]{Lemma} 
\newtheorem{pro}[thm]{Proposition}

\theoremstyle{definition} \newtheorem{dfn}[thm]{Definition}
\theoremstyle{remark} \newtheorem{rem}{Remark}

\DeclareMathOperator{\supp}{supp}

 \newcommand{\lb}{\langle}
\newcommand{\rb}{\rangle} \newcommand{\ls}{\lesssim}
\newcommand{\gs}{\gtrsim}

\begin{document}

\title[The Boson star equation]%
{The Boson star equation with \\ initial data of low regularity}
\author[S.~Herr]{Sebastian Herr}
\author[E.~Lenzmann]{Enno Lenzmann}
\thanks{
The first author acknowledges support from the German Research
Foundation, Collaborative Research Center 701.
}

\address{Universit\"at Bielefeld, Fakult\"at f\"ur Mathematik,
  Postfach 10 01 31, D-33501 Bielefeld, Germany}
\email{herr@math.uni-bielefeld.de}
\address{Universit\"at Basel, Departement f\"ur Mathematik und Informatik, Mathematisches Institut, Rheinsprung 21,
CH-4051 Basel, Schweiz} \email{enno.lenzmann@unibas.ch}
\begin{abstract}
The Cauchy problem for the $L^2$-critical boson star equation with initial data of low regularity in spatial dimension $d=3$ is studied. Local well-posedness in $H^s$ for $s > 1/4$ is proved. Moreover, for radial initial data, local well-posedness is established in $H^s$ for $s > 0$. Both results are shown to be almost optimal by providing complementary ill-posedness results.
\end{abstract}
\maketitle
\section{Introduction and Main Results}\label{sect:intro_main}
\noindent
We consider the initial value problem
\begin{equation}\label{eq:boson-star}
  \begin{split}
    i\partial_t u & = \sqrt{-\Delta+m^2} \, u- (|x|^{-1} \ast |u|^2)u \quad \text{in }(-T,T)\times \R^3,\\
    u(0,\cdot)&=\phi \in H^s(\R^3).
  \end{split}
\end{equation}
Here $\sqrt{-\Delta+m^2}$ is defined via its symbol $\sqrt{\xi^2+m^2}$ in Fourier space, where the constant $m \geq 0$ is a physical mass parameter, and the symbol $\ast$ denotes convolution in $\R^3$. The nonlinear dispersive evolution problem \eqref{eq:boson-star} arises as an effective equation describing the dynamics and gravitational collapse of relativistic boson stars; see \cite{MS,ES07,LY87,FJL07} and references therein. Given this physical background, we shall also refer to equation \eqref{eq:boson-star} as {\em the boson star equation} in the following.

We recall that the boson star equation exhibits the following conserved quantities of energy and $L^2$-mass, which are given by
$$
E(u(t)) = \frac  1 2 \int_{\R^3} \overline{u} \sqrt{-\Delta+m^2} \, u \, dx - \frac{1}{4} \int_{\R^3} (|x|^{-1} \ast |u|^2) |u|^2 \, dx ,
$$
$$
M(u(t)) = \int_{\R^3} |u|^2 \, dx.
$$
From these conservation laws, we see that the Sobolev space $H^{1/2}(\R^3)$ serve as the energy space for problem \eqref{eq:boson-star}. Furthermore, in the case of vanishing mass parameter $m=0$ in \eqref{eq:boson-star}, we have the scaling symmetry $$u(t,x) \mapsto u_\lambda(t,x) = \lambda^{3/2} u(\lambda t, \lambda x)$$ for fixed $\lambda > 0$. Clearly, the symmetry leaves the $L^2$-mass $M(u(t)) = M(u_\lambda(t))$ invariant. In this sense, the boson star equation exhibits the delicate feature of {\em $L^2$-criticality}.

Our aim is to determine minimal regularity requirements on the initial data in the Sobolev scale $H^s(\R^3)$ and $H^s_{rad}(\R^3)$ (subspace of radial functions), respectively, such that the initial value problem for the boson star equation is locally well-posed. Let us state our first main result.
\begin{thm}\label{thm:lwp} Let $m \geq 0$.
  \begin{enumerate}
\item Let $s> 1/4$. The initial
  value problem \eqref{eq:boson-star} is locally well-posed for
  initial data in $H^s(\R^3)$.
\item Let $s>0$. The initial
  value problem \eqref{eq:boson-star} is locally well-posed for initial data in $H^s_{rad}(\R^3)$.
\end{enumerate}
\end{thm}
In both cases, the notion of local well-posedness includes existence of solutions up to some time $T>0$, uniqueness of solutions in a certain subspace, persistence of initial regularity, and analytic dependence on the initial data.

A first well-posedness result for the boson star equation was obtained by the second author of this paper (see \cite{L07}) in $H^s$ for $s \geq 1/2$ by energy methods. Moreover, it was shown in \cite{L07} that global well-posedness holds in $H^{1/2}$ for initial data sufficiently small in $L^2$. In contrast to this, see \cite{FL07} for the existence of finite-time blowup solutions with smooth initial data that are large in $L^2$. Moreover, we refer the reader to \cite{CO06,CO07,COSS09} for further well-posedness results for the boson star equation with initial data in $H^s$ with $s$ slightly below $1/2$. 

Our second result shows that Theorem \ref{thm:lwp} is essentially sharp in the following sense.
\begin{thm}\label{thm:ill-c3} Let $m \geq 0$.
\begin{enumerate}
\item Let $s<\frac14$ and $T>0$. If the flow map $\phi \mapsto u$ exists (in a small neighbourhood of the origin) as a map from $H^s(\R^3)$ to $C([0,T],H^s(\R^3))$, it fails to be $C^3$ at the origin.
\item Let $s<0$ and $T>0$. If the flow map $\phi \mapsto u$ exists (in a small neighbourhood of the origin) as a map from $H^s_{rad}(\R^3)$ to $C([0,T],H^s_{rad}(\R^3))$, it fails to be $C^3$ at the origin.
\end{enumerate}
\end{thm}
Ill-posedness results similar to Theorem \ref{thm:ill-c3} have been
proved in \cite[Theorem 2]{MST01} for the Benjamin-Ono equation and in
\cite[pp. 155--158]{B97} for the Korteweg-de Vries equation. Compared to the $L^2$-critical NLS, it is interesting to note that the boson star equation \eqref{eq:boson-star} exhibits ill-posedness for non-radial data (in the sense given above) in a regularity class above the critical space $L^2$. Not so surprisingly, such a phenomenom of ill-posedness above scaling is much more akin to nonlinear wave equations (see \cite{L93}).  

We remark that both Theorem \ref{thm:lwp} and Theorem \ref{thm:ill-c3} remain true in the defocusing case, i.e.\ we may replace $-|x|^{-1}$ by $+|x|^{-1}$ in \eqref{eq:boson-star}.

For radial data in $L^2(\R^3)$ we prove the failure of uniform
continuity on balls, similar to the results in \cite{bkpsv96} for
nonlinear Schr\"odinger and generalized Benjamin-Ono equations:
\begin{thm}\label{thm:ill-unif}
Let $m \geq 0$ and $T>0$. If the flow map $\phi \mapsto u$ exists as a map from $L^2_{rad}(\R^3)$ to $C([0,T],L^2_{rad}(\R^3))$, it fails to be locally uniformly continuous.
\end{thm}
We refer the reader to Propositions \ref{pro:uniform} and \ref{pro:uniform2} for more precise statements. The proof of Theorem \ref{thm:ill-unif} utilizes the fact that \eqref{eq:boson-star} has solitary wave solutions $u(t,x) = e^{it \mu} Q_\mu(x)$, where $Q_\mu \in H^{1/2}(\R^3)$ solves the nonlinear ellipitic equation
$$
\sqrt{-\Delta + m^2} \, Q_\mu + \mu Q_\mu - ( |x|^{-1} \ast |Q_\mu|^2) Q_\mu = 0,
$$
with $\mu > 0$ given. In the case $m=0$, the exact scaling properties simplify the analysis considerably and we can adapt an argument in \cite{bkpsv96} in order to prove Theorem \ref{thm:ill-unif}. In contrast to this, the case $m>0$ breaking exact scaling deserves an additional discussion of the elliptic problem for $Q_\mu$.

When studying \eqref{eq:boson-star}, it might be convenient to add the linear term $mu$ by considering the function
$e^{-itm} u$ instead of $u$, which turns the equation \eqref{eq:boson-star} into
\begin{equation} \label{eq:boson-star2}
i \partial_t u = ( \sqrt{-\Delta + m^2} - m ) u - (|x|^{-1} \ast |u|^2 ) u  .
\end{equation}
Furthermore, by rescaling $u$, it suffices to consider either $m=0$ or $m=1$ in the proofs.

The paper is organized as follows: In Section \ref{sect:wp}, we prove Theorem \ref{thm:lwp} concerning local well-posedness. In Section \ref{sect:ill}, we prove Theorems \ref{thm:ill-c3} and Theorem \ref{thm:ill-unif} on ill-posedness. 
\section{Well-posedness}\label{sect:wp}
\noindent
Let us fix some notation.
We denote the Fourier transform of a tempered distribution $f$ both
by $\mathcal{F}f$ and $\widehat{f}$, and to indicate a partial Fourier
transform with respect to time and space variables we also write $\mathcal{F}_t f$
and $\mathcal{F}_x f$, respectively.

We denote dyadic numbers $\lambda=2^l: l \in \Z$ by greek letters.
Further, for $\lambda>1$ we define dyadic annuli
\[\Delta_\lambda:=\{\xi \in \R^3: \lambda/2 <|\xi|\leq 2\lambda\}, \,
\text{ and }\Delta_1:=\{\xi \in \R^3: |\xi|\leq 1\}.\]
We fix an even function $\beta_1 \in C_0^\infty((-2,2))$ s.t. $\beta_1(s)=1$ if $|s|\leq 1$, and define $\beta_\lambda(s)=\beta(s\lambda^{-1})-\beta(2s\lambda^{-1})$ for $\lambda> 1$. Next, we define the (smooth) Fourier localization operators in the standard way:
\[P_{\lambda}f:=f_\lambda:=\mathcal{F}^{-1} (\beta_\lambda(|\cdot|) \mathcal {F}f),\]
and $P_{\leq \lambda}=\sum_{\mu=1}^\lambda P_\mu$.
For measurable sets $S\subset \R^n$, let $\chi_S$ denote the sharp cutoff function, i.e. define $\chi_S(x)=1$ if $x\in S$ and zero otherwise. We set
\[P_{S}f=\mathcal{F}^{-1} (\chi_S \mathcal {F}f).\]

For fixed $m\geq 0$ let $U(t)$ be the linear propagator defined by
\[\mathcal{F}_x(U(t)\phi)(\xi)=e^{it \varphi_m(\xi)} \mathcal{F}_x\phi(\xi), \qquad \varphi_m(\xi)=m-\sqrt{m^2+ |\xi|^2}.\]

\begin{dfn}\label{dfn:xsb}
Let $s,b\in \R$. We define the space $X^{s,b}$ of tempered distributions $u \in \mathcal{S}(\R^{3+1})$ such that
\begin{equation}\label{eq:xsb}
\|u\|_{s,b}:=\left(\int_{\R^{3+1}}\lb \xi\rb^{2s} \lb\tau+|\xi| \rb^{2b} |\mathcal{F}u(\xi,\tau)|^2d\xi d\tau\right)^{\frac12}<\infty.
\end{equation}
Further, $X^{s,b}_{rad}$ denotes the closed subspace of spatially radial distributions.
\end{dfn}

Concerning well-posedness, it suffices to consider the case $m=0$ as long as we are considering short time scales only.
Indeed, $\lb\tau+|\xi| \rb \sim \lb\tau+\sqrt{m+|\xi|^2}-m\rb $ and the corresponding $\|\cdot\|_{s,b}$-norms are equivalent.

We first recall Strichartz type estimates for the wave equation. They have been studied systematically in \cite{S08}, see also references therein.
\begin{lem}\label{lem:str}
Let $b >\frac12$. \begin{enumerate}
\item For any ball $B_\mu$ with radius $\mu \geq 0$ and arbitrary
  center, and for all $u \in X^{\frac14,b}$ it holds
  \begin{equation}\label{eq:str1}
    \|P_{B_\mu} u
    \|_{L^4(\R \times \R^3)}\ls{} \mu^{\frac14}\|u\|_{\frac14,b}.
  \end{equation}
\item For all $u \in X^{\frac12,b}$ it holds
\begin{equation}\label{eq:str1b}
    \|u
    \|_{L^4(\R \times \R^3)}\ls{}\|u\|_{\frac12,b}.
  \end{equation}
\item For all $\mu \geq 0$ and for all $u_1,u_2 \in X^{\frac14,b}$ it holds
  \begin{equation}\label{eq:str2}
    \|P_{\mu} (\widetilde{u}_1 \widetilde{u}_2)
    \|_{L^2(\R \times \R^3)}\ls{} \mu^{\frac12}\|u_1\|_{\frac14,b}\|u_2\|_{\frac14,b},
  \end{equation}
\end{enumerate}
where $\widetilde{u}_j$ denotes either $u_j$ or $\overline{u}_j$.
\end{lem}
\begin{proof}
{\it Part i)} Due to \cite[Theorem 4.1]{S08} and the extension lemma \cite[Lemma 2.3]{GTV97}  we have
\begin{equation*}
    \|P_{B_\mu} P_\lambda u
    \|_{L^4(\R \times \R^3)}\ls{} \mu^{\frac14}\lambda^{\frac14}\|P_\lambda u\|_{0,b}.
  \end{equation*}
By Littlewood-Paley theory it follows that
\begin{align*}
\|P_{B_\mu} u\|_{L^4}\ls{} &\Big(\sum_{\lambda\geq 1}\| P_\lambda P_{B_\mu}u
    \|^2_{L^4(\R \times \R^3)}\Big)^{\frac12}\\\ls{}&\mu^{\frac14} \Big(\sum_{\lambda\geq 1}  \lambda ^{\frac12}\|P_\lambda u\|^2_{0,b}\Big)^{\frac12}\ls \mu^{\frac14}\|u\|_{\frac14,b}.
\end{align*}
{\it Part ii)} The Littlewood-Paley estimate and Part i) with $\mu\sim\lambda$ imply
\[
\|u\|_{L^4}\ls \Big(\sum_{\lambda\geq 1}\| P_\lambda u
    \|^2_{L^4(\R \times \R^3)}\Big)^{\frac12}\ls \Big(\sum_{\lambda\geq 1}\| P_\lambda u
    \|^2_{0,b}\Big)^{\frac12}\ls \|u\|_{0,b}.
\]
{\it Part iii)} We use almost orthogonality: Let us consider the collection of cubes $C_z=\mu z+[0,\mu)^3, z \in \Z^3$, which induce a disjoint covering of $\R^3$. We have
\begin{align*}
\|P_{\mu} (\widetilde{u}_1 \widetilde{u}_2)
    \|_{L^2(\R \times \R^3)}\ls{} \sum_{z,z'\in \Z^3}\|P_\mu( P_{C_z} \widetilde{u}_1 P_{C_{z'}}\widetilde{u}_2)\|_{L^2(\R \times \R^3)}
  \end{align*}
For each $z\in \Z^3$, only those $z'\in \Z^3$ with $|z-z'|\ls 1$ yield
a nontrivial contribution to the sum. Hence, by Cauchy-Schwarz and Part i) we obtain
\begin{align*}
&\|P_{\mu} (\widetilde{u}_1 \widetilde{u}_2)
    \|_{L^2(\R \times \R^3)}\\\ls{}&\mu^{\frac12}\Big(\sum_{z\in \Z^3}\|P_{B_{\sqrt{3}\mu} (\mu z)} u_1\|^2_{\frac14,b}\Big)^{\frac12}\Big(\sum_{z'\in \Z^3}\|P_{B_{\sqrt{3}\mu} (\mu z')} u_2\|^2_{\frac14,b}\Big)^{\frac12},
  \end{align*}
where $B_{\sqrt{3}\mu}(\mu z)$ denotes the ball
with radius $\sqrt{3}\mu$ and center $\mu z$, and the claim follows from $\sum_{z \in \Z^3}\chi^2_{B_{\sqrt{3}\mu} (\mu z)}\ls 1$.
\end{proof}
\begin{rem}
By (complex) interpolation of $\|u\|_{L^4_tL^4_x}\ls \|u\|_{\frac12,b}$ and $\|u\|_{L^2_tL^2_x}=\|u\|_{0,0}$ resp. $\|u\|_{L^\infty_t L^2_x}\ls\|u\|_{0,b}$ we obtain
\begin{align}
\label{eq:inter1}\|u\|_{L^{\frac83}_tL^{\frac83}_x}\ls &\|u\|_{\frac14,\frac{b}{2}}\\
\label{eq:inter2}\|u\|_{L^{8}_t L^{\frac83}_x}\ls &\|u\|_{\frac14,b}
\end{align}
if $b>\frac12$. The Sobolev embedding implies
\[
\|u\|_{L^4_tL^4_x}\ls \|D^{\frac34}u\|_{L^4_tL^2_x}\ls \|u\|_{\frac34,\frac14}.
\]
By interpolation with $\|u\|_{L^4_tL^4_x}\ls \|u\|_{\frac12,b}$ (for $b>\frac12$) we obtain the following: For any $\delta>0$ there exists $b<\frac12$ such that
\begin{equation}\label{eq:inter3}
\|u\|_{L^4_tL^4_x}\ls \|u\|_{\frac12+\delta,b}.
\end{equation}
Furthermore, there is one obvious consequence from \eqref{eq:str2} which we will use later: For $s>\frac14$, $b>\frac12$,
\begin{equation}\label{eq:str2-sub}
    \|P_{\mu} (\widetilde{u}_1 \widetilde{u}_2)
    \|_{L^2(\R \times \R^3)}\ls{} \mu^{\frac34-s}\|u_1\|_{s,b}\|u_2\|_{s,b},
  \end{equation}
\end{rem}

In the case of radial data there is the following improvement, which is an immediate consequence of
\cite[Theorem 2.6]{S08}.
\begin{lem}\label{lem:bil-str} Let $b >\frac12$. For all $\lambda\geq \mu \geq 0$, and for all $u_1,u_2 \in X^{0,b}_{rad}$
\begin{equation}\label{eq:bil-str}
\|P_{\mu} (P_{\lambda} \widetilde{u}_1 P_{\lambda} \widetilde{u}_2)
    \|_{L^2(\R \times \R^3)}\ls{} \mu \|u_1\|_{0,b}\|u_2\|_{0,b},
  \end{equation}
where $\widetilde{u}_j$ denotes either $u_j$ or $\overline{u}_j$.
\end{lem}
\begin{proof} It suffices to consider $\lambda\gg \mu$.
Decompose $u_j=u_{j,low}+u_{j,high}$ into low and high modulation,
i.e.
\begin{align*}
\supp
\mathcal{F}u_{j,low}\subseteq {}&\{(\tau,\xi): |\tau+|\xi||\leq
\mu\}\\
\supp
\mathcal{F}u_{j,high}\subseteq {}&\{(\tau,\xi): |\tau+|\xi||>
\mu\}
\end{align*} Because of \[\supp \mathcal{F}(P_{\mu}
(P_{\lambda}u_{1,low}P_{\lambda}u_{2,low}))\subseteq \{(\tau,\xi):
|\tau+|\xi||\ls \mu\}\] the estimate of \cite[Theorem 2.6]{S08} with
$r\sim \mu$ readily implies
\[
\|P_{\mu} (P_{\lambda} \widetilde{u}_{1,low} P_{\lambda} \widetilde{u}_{2,low})
    \|_{L^2(\R \times \R^3)}\ls{} \mu \|u_1\|_{0,b}\|u_2\|_{0,b}.
\]
On the other hand,
\begin{align*}
\|P_{\mu} (P_{\lambda} \widetilde{u}_{1,low} P_{\lambda} \widetilde{u}_{2,high})
 \|_{L^2(\R \times \R^3)}
\ls{}& 
\mu^{\frac32} \|P_{\lambda} \widetilde{u}_{1,low} P_{\lambda} \widetilde{u}_{2,high}
 \|_{L^2_t L^1_x}\\
\ls{}&
\mu^{\frac32} \|P_{\lambda} \widetilde{u}_{1,low}\|_{L^\infty_t L^2_x}\| P_{\lambda} \widetilde{u}_{2,high}
 \|_{L^2_t L^2_x}\\
\ls{}& \mu^{\frac32-b} \|u_1\|_{0,b}\|u_2\|_{0,b}.
\end{align*}
The same argument applies in the remaining cases, since at least one
factor has high modulation.
\end{proof}
\begin{rem}
By the Sobolev embeddding, we obtain
\[
\|P_{\mu} (P_{\lambda} \widetilde{u}_1 P_{\lambda} \widetilde{u}_2)
    \|_{L^2(\R \times \R^3)}\ls{} \mu^{\frac32}\|P_{\mu} (P_{\lambda} \widetilde{u}_1 P_{\lambda} \widetilde{u}_2)
    \|_{L^2_tL^1_x}\ls \mu^{\frac32} \|u_1\|_{0,\frac14}\|u_2\|_{0,\frac14}.
\]
Interpolation with \eqref{eq:bil-str} implies that for any $\delta>0$
there exists $b<\frac12$ such that
\begin{equation}\label{eq:inter4}
\|P_{\mu} (P_{\lambda} \widetilde{u}_1 P_{\lambda} \widetilde{u}_2)
    \|_{L^2(\R \times \R^3)}\ls{} \mu^{1+\delta} \|u_1\|_{0,b}\|u_2\|_{0,b}
\end{equation}
for radial $u_1,u_2$.
\end{rem}

Next, we are prepared to prove
\begin{pro}\label{pro:bil-est1}
Let $s> \frac14$. There exists $-\frac12<b'<0<\frac12<b\leq b'+1$ and $\delta>0$, such that
\begin{equation}\label{eq:bil-est1}
\||x|^{-1} \ast(u_1\overline{u}_2) \, u_3\|_{s,b'}\ls T^{\delta} \|u_1\|_{s,b}\|u_2\|_{s,b}\|u_3\|_{s,b}
\end{equation}
for all $u_j \in X^{s,b}$ with $\supp(u_j)\subset \{(t,x): |t|\leq
T\}$, $j=1,2,3$.
\end{pro}
\begin{proof}
Without loss we  may assume $\frac14< s<\frac38$.
We notice that in $\R^3$ convolution with $|x|^{-1}$ is (up to a multiplicative constant) the Fourier-multiplier $|D|^{-2}$ with symbol $|\xi|^{-2}$, which is locally integrable.
By duality, it suffices to prove
\begin{equation}\label{eq:bil-dual}
\begin{split}
I:=&\Big|\iint |D|^{-2}(u_1\overline{u}_2) u_3 \lb
D\rb^s\overline{u_4} dt dx \Big|\\
\ls & T^{\delta} \|u_1\|_{s,b}\|u_2\|_{s,b}\|u_3\|_{s,b}
\|u_4\|_{0,-b'}.
\end{split}
\end{equation}
It suffices to consider $u_j$ with non-negative Fourier-transform. Then, we may split the left hand side into two terms:
\begin{align*}
I\ls &\Big|\iint \lb D\rb^s|D|^{-2}(u_1\overline{u}_2) u_3
\overline{u_4} dt dx \Big| +\Big|\iint |D|^{-2}(u_1\overline{u}_2) \lb
D\rb^s u_3 \overline{u_4} dt dx \Big|\\
=:&I_1+I_2.
\end{align*}
Let us first discuss the contribution of $P_{\leq 2}( u_1\overline{u}_2)$:
For any $s_1,s_2\in \R$, the Hardy-Littlewood-Sobolev and Bernstein inequalities imply
\begin{align*}
&\Big|\iint \lb D\rb^{s_1}|D|^{-2}P_{\leq 2}(u_1\overline{u}_2) \lb D\rb^{s_2} u_3
\overline{u_4} dt dx\Big|\\
\ls & 
\||D|^{-2}P_{\leq 2} (u_1\overline{u_2}) \|_{L^2_tL^6_x}\|P_{\leq 2}(\lb D\rb^{s_2} u_3
\overline{u_4})\|_{L^2_tL^{\frac65}_x}\\
\ls & \|P_{\leq 2}(u_1\overline{u_2})\|_{L^2_tL^{\frac65}_x}\|\lb D\rb^{s_2}u_3\overline{u_4}\|_{L^2_tL^{1}_x}\\
\ls &\|u_1\|_{L^4_tL^2_x} \|u_2\|_{L^4_tL^2_x}\|\lb D\rb^{s_2}u_3\|_{L^4_tL^2_x} \|u_4\|_{L^4_tL^2_x}\\
\ls &\|u_1\|_{0,\frac14}\|u_2\|_{0,\frac14}\| u_3\|_{s_2,\frac14}\|u_4\|_{0,\frac14}.
\end{align*}
Thus, we may henceforth assume that $P_1(u_1\overline{u}_2)=P_1(u_3\overline{u}_4)=0$.
First, we consider the contribution of $I_1$:
\[
I_1\leq \|\lb D\rb^s|D|^{-\frac78}(u_1\overline{u}_2)\|_{L^2} \||D|^{-\frac98}(u_3\overline{u}_4)\|_{L^2}=:I_{11}\cdot I_{12}
\]
On the one hand, using \eqref{eq:str2} we obtain
\begin{align*}
I_{11} \ls & \sum_{\mu \geq 2}\mu^{s-\frac78}\|P_\mu(u_1\overline{u}_2)\|_{L^2}
\ls \sum_{\mu \geq 2}\mu^{s-\frac38}\|u_1\|_{\frac14,b}\|u_2\|_{\frac14,b}\\\ls &\|u_1\|_{\frac14,b}\|u_2\|_{\frac14,b}
\end{align*}
On the other hand, the Hardy-Littlewood-Sobolev inequality and \eqref{eq:inter2} yield
\begin{equation*}
I_{12} \ls  \|u_3\overline{u}_4\|_{L^2_tL^{\frac87}_x}\ls \|u_3\|_{L^{8}_tL^{\frac83}_x}\|u_4\|_{L^{\frac83}_tL^{2}_x}
\ls \|u_3\|_{\frac14,b}\|u_4\|_{0,\frac18}.
\end{equation*}
For $-b'> \frac18$, $b>\frac12$, this implies
\begin{equation}
I_1\ls I_{11}+I_{12}\ls  T^{\delta} \|u_1\|_{\frac14,b}\|u_2\|_{\frac14,b} \|u_3\|_{\frac14,b}\|u_4\|_{0,-b'}.
\end{equation}

Finally, we turn to $I_2$: By Cauchy-Schwarz, \eqref{eq:str2-sub} and Bernstein we obtain
\begin{align*}
I_2\ls & \sum_{\mu \geq 2}\mu^{-2}\|P_\mu (u_1\overline{u}_2)\|_{L^2}\|P_\mu (\lb D\rb^s u_3\overline{u}_4)\|_{L^2}\\
\ls&\|u_1\|_{s,b}\|u_2\|_{s,b} \sum_{\mu \geq 2}\mu^{\frac14-s}\|P_\mu (\lb D\rb^s u_3\overline{u}_4)\|_{L^2_tL^1_x}\\
\ls&\|u_1\|_{s,b}\|u_2\|_{s,b} \sum_{\mu \geq 2}\mu^{\frac14-s}\|\lb D\rb^s u_3\|_{L^4_tL^2_x}\|u_4\|_{L^4_tL^2_x}\\
\ls& T^\delta\|u_1\|_{s,b}\|u_2\|_{s,b}\|u_3\|_{s,b}\|u_4\|_{0,-b'}
\end{align*}
if $b'<-\frac14$ and $b>\frac12$, $s>\frac14$.
\end{proof}
Next, we consider the radial case.
\begin{pro}\label{pro:bil-est2}
Let $s>0$. There exists $-\frac12<b'<0<\frac12<b\leq b'+1$ and $\delta>0$, such that
\begin{equation}\label{eq:bil-est2}
\| |x|^{-1} \ast(u_1\overline{u}_2) \, u_3\|_{s,b'}\ls T^{\delta} \|u_1\|_{s,b}\|u_2\|_{s,b}\|u_3\|_{s,b}
\end{equation}
for all $u_j\in X^{s,b}_{rad}$ with $\supp(u_j)\subset \{(t,x): |t|\leq T\}$.
\end{pro}

\begin{proof}
As above, by duality, it suffices to prove
\begin{equation}\label{eq:bil-dual2}
\begin{split}
I:=&\Big|\iint |D|^{-2}(u_1\overline{u}_2) u_3 \lb D\rb^s
\overline{u_4} dt dx \Big|\\
\ls& T^{\delta} \|u_1\|_{s,b}\|u_2\|_{s,b}\|u_3\|_{s,b}
\|u_4\|_{0,-b'}
\end{split}
\end{equation}
for $0<s<\frac12$.
It suffices to consider $u_j$ with non-negative Fourier-transform. Then, we may split the left hand side into two terms:
\begin{align*}
I\ls& \Big|\iint \lb D\rb^s |D|^{-2}(u_1\overline{u}_2) u_3
\overline{u_4} dt dx \Big| +\Big|\iint |D|^{-2}(u_1\overline{u}_2) \lb
D\rb^s u_3 \overline{u_4} dt dx \Big|\\
=:&I_1+I_2.
\end{align*}
The argument in the proof of Proposition \ref{pro:bil-est1} shows that we may henceforth assume that $P_1(u_1\overline{u}_2)=P_1(u_3\overline{u}_4)=0$.
First, we consider the contribution of $I_1$:
\[
I_1\leq \|\lb D\rb^s|D|^{-1}(u_1\overline{u}_2)\|_{L^2}\||D|^{-1}(u_3\overline{u}_4)\|_{L^2} \\
\ls I_{11}\cdot I_{12}
\]
We have (dyadic summation)
\begin{align*}
I_{11}
\ls &\sum_{\lambda_1\ll \lambda_2} \lambda_2^{s-1}\|u_{1,\lambda_1}\overline{u}_{2,\lambda_2}\|_{L^2}+\sum_{\lambda_1\gg \lambda_2} \lambda_1^{s-1}\|u_{1,\lambda_1}\overline{u}_{2,\lambda_2}\|_{L^2}
\\&+\sum_{\mu \ls \lambda}
\mu^{s-1}\|P_{\mu}(u_{1,\lambda}\overline{u}_{2,\lambda})\|_{L^2}
\end{align*}
We start estimating the first term by using H\"older's inequality and \eqref{eq:str1b}:
\begin{align*}
\sum_{\lambda_1\ll \lambda_2}
\lambda_2^{s-1}\|u_{1,\lambda_1}\overline{u}_{2,\lambda_2}\|_{L^2}
\ls&\sum_{\lambda_1\ll \lambda_2} \lambda_2^{s-\frac12}\lambda_1^{\frac12}\|u_{1,\lambda_1}\|_{0,b}\|u_{2,\lambda_2}\|_{0,b}\\
\ls &\sum_{k=0}^\infty \sum_{0\leq l \leq k} 2^{-l(\frac12-s)}\|u_{1,2^{k-l}}\|_{s,b}\|u_{2,2^k}\|_{0,b}\\
\ls &\|u_{1}\|_{s,b}\|u_{2}\|_{0,b}.
\end{align*}
For the second term, we obtain the same result.
The third term is estimated by using \eqref{eq:bil-str}:
\begin{align*}
\sum_{\mu \ls \lambda}
\mu^{s-1}\|P_{\mu}(u_{1,\lambda}\overline{u}_{2,\lambda})\|_{L^2}
\ls & \sum_{\mu \ls \lambda}
\mu^{s}\|u_{1,\lambda}\|_{0,b}\|u_{2,\lambda}\|_{0,b}\\
\ls &\|u_{1}\|_{s,b}\|u_{2}\|_{0,b}.
\end{align*}

Next, we decompose
$I_{12}=J_1+J_2+J_3$.
Using \eqref{eq:inter3} with $\delta=s$ we see that
\begin{align*}
J_1:=&\sum_{\lambda_3\ll \lambda_4} \lambda_4^{-1}\|u_{3,\lambda_3}\overline{u}_{4,\lambda_4}\|_{L^2}
\ls \sum_{\lambda_3\ll \lambda_4} \lambda_3^{\frac12}\lambda_4^{-\frac12+s}\|u_{3,\lambda_3}\|_{0,b}\|u_{4,\lambda_4}\|_{0,-b'}\\
\ls& \|u_{3}\|_{s,b}\|u_{4}\|_{0,-b'}
\end{align*}
for some $b'>-\frac12$.
Similarly,
\[
J_2:= \sum_{\lambda_3\gg \lambda_4} \lambda_3^{-\frac12}\lambda_4^{\frac12+s}\|u_{3,\lambda_3}\|_{0,b}\|u_{4,\lambda_4}\|_{0,-b'}\ls \|u_{3}\|_{s,b}\|u_{4}\|_{0,-b'}.
\]
Using \eqref{eq:inter4} we obtain
\begin{align*}
J_3:=&\sum_{\mu \ls \lambda} \mu^{-1}\|P_{\mu}(u_{3,\lambda}\overline{u}_{4,\lambda})\|_{L^2}
\ls \sum_{\mu \ls \lambda} \mu^{s}\|u_{3,\lambda}\|_{0,b}\|u_{4,\lambda}\|_{0,-b'}\\\ls&  \|u_{3}\|_{s,b}\|u_{4}\|_{0,-b'}
\end{align*}
for some $b'>-\frac12$. All in all, we have proved
\[
I_1\ls \|u_{1}\|_{s,b}\|u_{2}\|_{0,b} \|u_{3}\|_{s,b}\|u_{4}\|_{0,-b'}.
\]
Concerning $I_2$, we obtain
\[
I_2\leq \|\lb D\rb^s |D|^{-1}(u_1\overline{u}_2)\|_{L^2}\|\lb D\rb^{-s}|D|^{-1}(\lb D\rb^s u_3 \overline{u}_4)\|_{L^2} \\
\ls I_{21}\cdot I_{22}
\]
The estimate for $I_{11}$ above also applies to $I_{21}$. As above, we decompose
$I_{22}=K_1+K_2+K_3$. Trivial modifications of the arguments above yield
\[
I_{22}\ls \|u_{3}\|_{s,b}\|u_{4}\|_{0,-b'},
\]
and altogether we obtain
\[
I_2\ls \|u_{1}\|_{s,b}\|u_{2}\|_{0,b} \|u_{3}\|_{s,b}\|u_{4}\|_{0,-b'}.
\]
By replacing $\|u_{4}\|_{0,-b'}$ by $\|u_{4}\|_{0,-b''}$ for some $-\frac12<b''<b'$ shows that one can squeeze out a factor $T^\delta$.
\end{proof}

Finally, we explain how Propositions \ref{pro:bil-est1} and
\ref{pro:bil-est2} imply Theorem \ref{thm:lwp}: The general idea -- due to
Bourgain \cite{B93} -- is well-known by now,
see e.g. \cite{GTV97} and references therein for more details. The
basic idea is to solve the integral equation
\begin{equation}\label{eq:duhamel}
u(t)=\psi_TU(t)\phi +i\psi_T\int_0^t U(t-\tau) \big(|x|^{-1} \ast |U(\tau) \phi|^2 U(\tau) \phi \big)
d\tau
\end{equation}
with the smooth cutoff $\psi_T(t)=\beta_1(t/T)$ and given initial data $\phi$ by the contraction
mapping principle.
For $b>\frac12$ \cite[(2.19)]{GTV97} implies
\[\|\psi_TU(t)\phi \|_{s,b}\ls T^{\frac12-b}\|\phi\|_{H^s} \]
Also,  \cite[Lemma 2.1]{GTV97} implies that
\[\| \psi_T\int_0^t U(t-\tau)f(\tau)d\tau \|_{s,b}\ls  T^{1-b+b'} \|f\|_{s,b'} .\]
if $-\frac12<b'<0<\frac12<b\leq b'+1$.
Now, due to Propositions \ref{pro:bil-est1} and
\ref{pro:bil-est2}, it is an easy exercise to verify that the right
hand side in \eqref{eq:duhamel} defines a contraction in an
appropriate closed ball in $X^{s,b}\subset C([0,T];H^s(\R^3))$ and $X^{s,b}_{rad}\subset C([0,T];H^s_{rad}(\R^3))$, respectively. Hence, it has a fixed
point. Uniqueness in $X^{s,b}$ resp. $X^{s,b}_{rad}$ and smooth (real analytic) dependence on the initial data are immediate consequences.
\section{Ill-posedness}\label{sect:ill}
\subsection{Proof of Theorem \ref{thm:ill-c3}}\label{subsect:ill-c3}
\noindent
As explained in \cite[Section 2.2]{MST01} (for the quadratic Benjamin-Ono equation), it suffices to prove the following:
\begin{pro}\label{pro:fail}
For fixed $0<t\leq 1$ and $s<\frac14$ the inequality
\begin{equation}\label{eq:fail}
\Big\|\int_0^t U(t-\tau)
\big(|x|^{-1} \ast |U(\tau) \phi|^2 U(\tau) \phi \big)
d\tau \Big\|_{H^s(\R^3)}\ls \|\phi \|^3_{H^s(\R^3)}
\end{equation}
fails to hold for all $\phi \in H^s(\R^3)$ (in any neighborhood of the origin).
\end{pro}
\begin{proof}
Let $1\leq \mu \ll \lambda$. We will choose $\mu=\mu(\lambda)=\delta \lambda^{\frac12}$ for fixed $0<\delta\ll 1$. Define the cube \[W_\lambda^\pm =\{(\xi_1,\xi_2,\xi_3) \in \R^3 \colon |\xi_1\mp \lambda|\leq \mu , \, |\xi_2|, |\xi_3|\leq \mu \},\]
which is centered at $\pm \lambda e_1$ with sidelength $2\mu$. Let $\phi$ be the inverse Fourier transform of the characteristic function $\chi_{W^+_\lambda}$ of $W^+_\lambda$. Obviously, $\|\phi\|_{H^s(\R^3)}\approx \mu^{\frac32} \lambda^s$.

Next, we consider
\[
F_t(\xi):=\mathcal{F}_x\Big( \int_0^t U(t-\tau)
\big( |x|^{-1} \ast |U(\tau) \phi|^2 U(\tau) \phi \big)
d\tau \Big)(\xi).
\]
Our aim is to prove that for $0<t\ll 1$ and all $\xi \in \tfrac14 W^+_\lambda$
\begin{equation}\label{eq:red}
 |F_t(\xi)|\gs |t| \mu^{4}.
\end{equation}
Assuming that \eqref{eq:red} holds, the claim follows since the validity of \eqref{eq:fail} implies
\[ |t| \mu^{\frac{11}{2}}\lambda^s \ls \|\lb \xi \rb^s F_t(\xi)\|_{L^2_\xi(\R^3)}\ls \mu^{\frac92} \lambda^{3s}
\]
which for $\mu=\delta \lambda^{\frac12}$ is equivalent to
\[
|t| \delta \ls  \lambda^{2s-\frac12},
\]
which can hold for fixed $t,\delta>0$ and $\lambda \to \infty$ only if $s\geq \frac14$. Hence, it suffices to establish \eqref{eq:red}: Similarly to  \cite[p.985]{MST01} we compute
\begin{align*}
&F_t(\xi)\\
=&c e^{it \varphi_m(\xi)}\int_{\R^3}\int_{\R^3} \frac{(e^{it r_m(\xi_1,\xi_2,\xi)}-1)}{ir_m(\xi_1,\xi_2,\xi)}\frac{\chi_{W^+_{\lambda}}(\xi_1)\chi_{W^-_{\lambda}}(\xi_2)}{|\xi_1+\xi_2|^{2}}\chi_{W^+_{\lambda}}(\xi-\xi_1-\xi_2) d\xi_1d\xi_2,
\end{align*}
where
\[
r_m(\xi_1,\xi_2,\xi)=\varphi_m(\xi_1)-\varphi_m(\xi_2)+\varphi_m(\xi-\xi_1-\xi_2)-\varphi_m(\xi).
\]
We notice that in the domain of integration we have
\[
||\xi_1|-|\xi_2|+|\xi-\xi_1-\xi_2|-|\xi||\ls \frac{\mu^2}{\lambda},
\]
hence, for each fixed $m \geq 0$,
\[
|r_m(\xi_1,\xi_2,\xi)|\ls \frac{1}{\lambda}+\frac{\mu^2}{\lambda}\ll 1 \text{ (by choosing $\delta>0$ small enough)}.
\]
Therefore, if $\xi \in \frac14 W^+_{\lambda}$ we have
\[
|F_t(\xi)|\gs |t| \int_{\frac14 W^+_{\lambda}}\int_{\frac14 W^-_{\lambda}} |\xi_1+\xi_2|^{-2} d\xi_1d\xi_2\gs |t| \mu^4,
\]
and the proof is complete.
Notice that (by multiplying the above function $\phi$ by a small fixed parameter) we can provide such a counterexample in any neighbourhood of the origin.
\end{proof}

\begin{pro}\label{pro:fail-rad}
For fixed $T>0$ and $s<0$ the inequality
\begin{equation}\label{eq:fail-rad}
\sup_{t \in [0,T]}\Big\|\int_0^t U(t-\tau)
\big(|x|^{-1} \ast |U(\tau) \phi|^2 U(\tau) \phi \big)
d\tau \Big\|_{H^s(\R^3)}\ls \|\phi \|^3_{H^s(\R^3)}
\end{equation}
fails to hold for all radial $\phi \in H^s(\R^3)$ (in any neighborhood of the origin).
\end{pro}
\begin{proof}
Let $T>0$ be fixed. For $ \lambda\gg 1$ we define the annulus \[A_\lambda =\{(\xi_1,\xi_2,\xi_3) \in \R^3 \colon \lambda\leq |\xi|\leq 2\lambda \},\]
and let $\phi$ be the inverse Fourier transform of the characteristic function $\chi_{A_\lambda}$ of $A_\lambda$. Obviously, $\phi$ is radial and $\|\phi\|_{H^s(\R^3)}\approx \lambda^{s+\frac32}$.
As above we consider
\[
F_t(\xi):=\mathcal{F}_x\Big(\int_0^t U(t-\tau)
\big(|x|^{-1} \ast |U(\tau) \phi|^2 U(\tau) \phi \big)
d\tau \Big)(\xi).
\]
For $t=\delta\lambda^{-1}$ with $0<\delta\ll 1$ (such that $t<T$) and $\xi \in \tfrac14 A_\lambda$ we will prove
\begin{equation}\label{eq:red-rad}
 |F_t(\xi)|\gs \delta \lambda^{3}.
\end{equation}
The estimate \eqref{eq:fail-rad} in conjuction with \eqref{eq:red-rad} implies
\[ \delta \lambda^{s+\frac{9}{2}} \ls \|\lb \xi \rb^s F_t(\xi)\|_{L^2_\xi(\R^3)}\ls \lambda^{3s+\frac92}
\]
which can hold for fixed $\delta>0$ and $\lambda \to \infty$ only if $s<0$.
It remains to prove \eqref{eq:red-rad}: As above (cp. \cite[p.985]{MST01}) we compute
\begin{align*}
&F_t(\xi)\\
=&c e^{it \varphi_m(\xi)}\int_{\R^3}\int_{\R^3} \frac{(e^{it r_m(\xi_1,\xi_2,\xi)}-1)}{ir_m(\xi_1,\xi_2,\xi)}\frac{\chi_{A_{\lambda}}(\xi_1)\chi_{A_{\lambda}}(\xi_2)}{|\xi_1+\xi_2|^{2}}\chi_{A_{\lambda}}(\xi-\xi_1-\xi_2) d\xi_1d\xi_2,
\end{align*}
where
\[
r_m(\xi_1,\xi_2,\xi)=\varphi_m(\xi_1)-\varphi_m(\xi_2)+\varphi_m(\xi-\xi_1-\xi_2)-\varphi_m(\xi).
\]
Obviously, in the domain of integration we have
\[
|t r_m(\xi_1,\xi_2,\xi)|\ls |t \lambda|\ll 1.
\]
Therefore, if $\xi \in \frac14 A_{\lambda}$ and $t=\delta\lambda^{-1}$ we have
\[
|F_t(\xi)|\gs \delta \lambda^{-1} \int_{\frac14 A_{\lambda}}\int_{A_{\lambda}} |\xi_1+\xi_2|^{-2} d\xi_1d\xi_2\gs  \delta \lambda^3,
\]
and the proof is complete.
As above, we can provide such a counterexample in any neighbourhood of the origin by multiplying $\phi$ by a small parameter.
\end{proof}

\subsection{Proof of Theorem
  \ref{thm:ill-unif}: Failure of Uniform Continuity}\label{subsect:ill-unif}
By exploiting the fact that \eqref{eq:boson-star} exhibits solitary wave solutions, we show failure of (local) uniform continuity of the solution map $\phi \mapsto u(t)$ in $L^2$. In the case $m=0$, the exact scaling symmetry of \eqref{eq:boson-star} simplifies the analysis considerably. From \cite[Appendix A.2]{L07} (see also \cite{LY87}) we recall the existence of radial ground state solutions $Q \in H^{\frac 1 2}(\R^3)$ of
\begin{equation} \label{eq:boson_ground}
\sqrt{-\Delta} \, Q + Q - (|x|^{-1} \ast |Q|^2) Q = 0.
\end{equation}
We divide this subsection by treating first the massless case $m=0$, followed by a discussion of the more complicated situation when $m>0$ hold in \eqref{eq:boson-star}. The following result concerns the massless case. To treat the case $m>0$, we need more elaborate arguments worked out below. 
\begin{pro} \label{pro:uniform}
Suppose that $m =0$ holds in \eqref{eq:boson-star} and let $t>0$. Then the map $\phi \mapsto u(t)$ fails to be uniformly continuous for initial data in the set
\[M= \{ \phi \in L^2(\R^3) \colon \phi \text{ radial }, \; \| \phi \|^2_{L^2} \geq \| Q \|^2_{L^2} \}\]
with respect to the $L^2$-norm.
\end{pro}

\begin{proof}
We adapt the arguments in \cite{bkpsv96} to our setting here.
Suppose that $m=0$ holds in \eqref{eq:boson-star} and suppose that $Q$ is a ground state solution \eqref{eq:boson_ground}. By scaling, we have that $Q_\mu(x) = \mu^{\frac 3 2} Q(\mu x)$ with $\mu > 0$ solves
\begin{equation*}
\sqrt{-\Delta} \, Q_\mu + \mu Q_\mu - (|x|^{-1} \ast |Q_\mu|^2) Q_\mu = 0.
\end{equation*}
Note that $\| Q_\mu \|_{L^2} = \| Q \|_{L^2}$ reflecting the $L^2$-criticality of \eqref{eq:boson-star}. 

For $\mu_1, \mu_2 > 0$ we consider the solutions
$$
u_{\mu_1}(t,x) = e^{i t \mu_1} Q_{\mu_1}(x), \quad u_{\mu_2}(t,x) = e^{it \mu_2} Q_{\mu_2}(x). 
$$
and we set
$$
I_{\mu_1, \mu_2}(t) = \| u_{\mu_1}(t) - u_{\mu_2}(t) \|_{L^2}^2.
$$
Note that
\begin{align*}
I_{\mu_1, \mu_2} (t) & = \| Q_{\mu_1} \|_{L^2}^2 + \| Q_{\mu_2} \|_{L^2}^2 - 2 \mbox{Re} \, \langle e^{it \mu_1} Q_{\mu_1}, e^{it \mu_2} Q_{\mu_2} \rangle \\
& = 2 \| Q \|_{L^2}^2 - 2 \mbox{Re} \, \left \{ e^{it (\mu_1 - \mu_2)} \left ( \frac{\mu_2}{\mu_1} \right )^{3/2} \int_{\R^3} Q(x) Q \left ( \left ( \frac{\mu_1}{\mu_2} \right ) x \right ) \, dx \right \}  \\
& = 2 \| Q \|_{L^2}^2 - 2 \cos (t (\mu_1 - \mu_2))  \left ( \frac{\mu_2}{\mu_1} \right )^{3/2} \int_{\R^3} Q(x) Q \left ( \left ( \frac{\mu_1}{\mu_2} \right ) x \right ) \, dx 
\end{align*}
Let $t >0$ be fixed and choose the sequences
$$
\mu_1(n) =  \frac{\pi}{2t} (n+1)^2, \quad \mu_2(n) = \frac{\pi}{2 t} n^2 , \quad \mbox{with $n \in \mathbb{N}$.}
$$ 
This choice implies that $t (\mu_1(n)-\mu_2(n)) = \frac{\pi}{2} (2n+1)$ and thus $\cos (t (\mu_1(n)-\mu_2(n)) ) = 0$. Therefore,
$$
I_{\mu_1(n), \mu_2(n)} (t) = 2 \| Q \|_{L^2}^2 \quad \mbox{for all $n \in \mathbb{N}$}.
$$
On the other hand, since $\mu_1(n)/\mu_2(n) \to 1$ as $n \to +\infty$, we easily see that
$$
\lim_{n \to +\infty} I_{\mu_1(n), \mu_2(n)} (0) = 0.$$
But this show that the solution $u(t)$ cannot depend in a uniformly continuous way on the initial datum $\phi$ in the $L^2$-topology. 

This completes the proof of Proposition \ref{pro:uniform}. \end{proof}

We now turn to the case when $m>0$ holds in \eqref{eq:boson-star}. Without loss of generality, we assume that
$
m=1
$
throughout the following. To derive an illposedness result that is analogue to Proposition \ref{pro:uniform}, we utilize solitary wave solutions $u(t,x) = e^{it \mu} Q_\mu(x)$, where $Q_\mu \in H^{1/2}(\R^3)$ are radial positive solutions to
\begin{equation} \label{eq:Qmu}
\sqrt{-\Delta +1} \, Q_\mu + \mu Q_\mu - (|x|^{-1} \ast |Q_\mu|^2) Q_\mu = 0.
\end{equation}
Here $\mu > 0$ can be chosen arbitrarily; in contrast to \cite{LY87}, where $\mu$ arises as a Lagrange multiplier.  However, we need some auxiliary results about a class of $Q_\mu$ that arise from a suitable variational problem. In particular, we are ultimately interested in the limit of large $\mu >0$ and we will show that 
\begin{equation*}
R_\mu(x) = \mu^{-3/2} Q_\mu(\mu^{-1} x)
\end{equation*} 
converges strongly in $L^2$ (up to subsequences) as $\mu \to \infty$ to some positive radial solution $R \in H^{1/2}$ that solves $\sqrt{-\Delta} \, R + R - (|x|^{-1} \ast |R|^2) R =0$. 

Note that $R_\mu \in H^{1/2}$ solves 
\begin{equation} \label{eq:Rmu}
\sqrt{-\Delta + \mu^{-2} } \, R_\mu + R_\mu - (|x|^{-1} \ast |R_\mu|^2) R_\mu = 0
\end{equation} 
if and only if $Q_\mu \in H^{1/2}$ solves \eqref{eq:Qmu}. In order to construct solutions $R_\mu$ (for any $\mu > 0$ given), we define the functional $F_\mu : H^{1/2}(\R^3) \to \R$ given by
\begin{equation*}
F_\mu( u ) = \int_{\R^3} \overline{u} \sqrt{-\Delta + \mu^{-2}} \, u \, dx + \int_{\R^3} |u|^2 \, dx .
\end{equation*}
We consider the following minimization problem
\begin{equation} \label{eq:Fmin}
F^*_\mu = \inf \{ F_\mu(u) : u \in H^{1/2}(\R^3), \; V(u) = 1 \},
\end{equation}
where
\begin{equation*}
V(u) = \int_{\R^3} (|x|^{-1} \ast |u|^2) |u|^2 \, dx .
\end{equation*}
By interpolation inequalities, it is easy to see that $F^*_\mu > 0$ holds. In fact, the infimum is attained.

\begin{lem} \label{lem:Tmu}
For any $\mu > 0$, there exists a radial positive minimizer $T_\mu \in H^{1/2}(\R^3)$ for problem \eqref{eq:Fmin}. Moreover, the function $T_\mu$ satisfies
$$
\sqrt{-\Delta + \mu^{-2}} \, T_\mu + T_\mu - \theta (|x|^{-1} \ast |T_\mu|^2) T_\mu = 0,
$$
where $\theta = \theta(T_\mu) > 0$ is some Lagrange multiplier. In particular, the function $R_\mu = \theta^{1/2} T_\mu \in H^{1/2}(\R^3)$ is a positive radial solution of \eqref{eq:Rmu}. 
\end{lem}

\begin{proof}
This follows from standard variational arguments. We provide a brief sketch of the main steps, where we use rearrangement inequalities to gain compactness of minimizing sequences. 

For $\sigma \geq 1$ given, we embed \eqref{eq:Fmin} into the family of variational problems
$$
I_\mu(\sigma) = \inf \{ F_\mu(u) : u \in H^{1/2}, \; V(u) = \sigma \} . 
$$ 
Clearly, we have that $I_\mu(\sigma) \geq 0$. Moreover, by scaling, we readily check that $I_\mu(\sigma) = \sigma^{1/2} I_\mu(1)$ and hence $I_\mu(\sigma) \geq I_\mu(1)$ for $\sigma \geq 1$. Now, by rearrangement inequalities (see, e.\,g., \cite{LL01}) we have that $F_\mu(u^*) \leq F_\mu(u)$ and $V(u^*) \geq V(u)$, where $u^*$ denotes the symmetric-decreasing rearrangement of $u  \in H^{1/2}$. Therefore, any minimizing sequence $(u_n)$ for problem \eqref{eq:Fmin} can be replaced by $(u_n^*)$ without loss of generality. Hence we assume $u_n^* = u_n$ from now on. Since $F_\mu(u_n) \to F_\mu^*$, we have that $\| u_n \|_{H^{1/2}} \lesssim 1$. Thus, by passing to a subsequence, we can assume that $u_n \rightharpoonup u_*$ weakly in $H^{1/2}$ and $u_n \to u_*$ strongly in $L^{p}_{\mathrm{loc}}$ for $p \in [1, 3)$. Furthermore, by using that $\| u_n \|_{L^2} \lesssim 1$ and the fact that $u_n = u_n(|x|)$ are radial and monotone decreasing in $|x|$, we deduce the pointwise bound $|u_n(x)| \lesssim |x|^{-3/2}$. Using this decay bound together with $u_n \to u_*$ strongly in $L^{12/5}_{\mathrm{loc}}$, we deduce that $u_n \to u_*$ strongly in $L^{12/5}$. Thus, by the Hardy-Littlewood-Sobolev inequality, this implies that $V(u_n) \to V(u_*)=1$. In particular, we have that $u_* \not \equiv 0$ holds. Next, by the lower semi-continuity of $F_\mu$ with respect to weak convergence in $H^{1/2}$, we deduce that $\lim_{n \to \infty} F_\mu(u_n) \geq F_\mu(u_*)$. Hence $T_\mu:=u_* \geq 0$ is a radial and nonnegative minimizer for problem \eqref{eq:Fmin} and satisfies the corresponding Euler-Lagrange equation with some mutliplier $\theta \in \R$. The positivity of $\theta > 0$ follows from integrating the Euler-Lagrange equation against $T_\mu$. In fact, we easily see that $\theta \leq 0$ is not possible for $T_\mu \not \equiv 0$.  Finally, we note that $T_\mu(x) > 0$ is in fact positive by adapting an argument in \cite{L07}.\end{proof}

Next we derive uniform bounds for the minimizers $T_\mu$ given above when taking in the limit $\mu \to \infty$. Hence, without loss generality, we consider the case $\mu \geq 1$, say. We have the following uniform $L^2$ bounds.

\begin{lem} \label{lem:Rmu_bounds}
Suppose that $\mu \geq 1$ and let $R_\mu = \theta^{1/2} T_\mu \in H^{1/2}(\R^3)$ be given as in Lemma \ref{lem:Tmu} above. Then we have the uniform bounds
$$
1 \lesssim \| R_\mu \|_{L^2} \leq \| R_\mu \|_{H^{1/2}} \lesssim 1 .
$$
\end{lem}

\begin{proof}
We derive approriate bounds for $T_\mu \in H^{1/2}$ and $\theta = \theta(T_\mu) > 0$ as follows.

First, we claim that
\begin{equation} \label{ineq:uni1}
\| T_\mu \|_{H^{1/2}} \lesssim 1.
\end{equation}
To see this, we simply note that $T_\mu$ minimizes problem \eqref{eq:Fmin}. Therefore, by taking some fixed function $w \in H^{1/2}$ such that $V(w) =1$, we deduce that
\begin{align*}
F_\mu(T_\mu) & \leq F_\mu(w) \leq \int_{\R^3} \overline{w} \sqrt{-\Delta} w \, dx + (1+\mu^{-1}) \int_{\R^3} |w|^2 \, dx \\
& \lesssim 1 + \mu^{-1} \lesssim 1, 
\end{align*}
using the operator inequality $\sqrt{-\Delta + \mu^{-2}} \leq \sqrt{-\Delta} + \mu^{-1}$, which directly follows in Fourier space. From this we deduce that the bound \eqref{ineq:uni1} holds true. 

Next, we derive an upper bound for $\theta >0$ as follows. By integrating the Euler-Lagrange equation for $T_\mu$ against $T_\mu$, we deduce that
\begin{equation} \label{ineq:uni2}
 \theta = F_\mu(T_\mu) \lesssim 1,
\end{equation}
using that $V(T_\mu) = 1$. Combining \eqref{ineq:uni1} and \eqref{ineq:uni2}, we find that $R_\mu = \theta^{1/2} T_\mu$ satisfies
\begin{equation*}
\| R_\mu \|_{L^2}  \leq \| R_\mu \|_{H^{1/2}} \lesssim \theta^{1/2} \| T_\mu \|_{H^{1/2}} \lesssim 1 .
\end{equation*}

It remains to show the uniform lower 
\begin{equation} \label{ineq:uni3}
\| R_\mu \|_{L^2} \gtrsim 1.
\end{equation}
Indeed, this can be seen by exploiting a Hardy type inequality as follows. Let $H$ denote the self-adjoint operator
$$
H = \sqrt{-\Delta+\mu^{-2}} -  (|x|^{-1} \ast |R_\mu|^2).
$$
Note that $-1$ is an eigenvalue of $H$, since $H R_\mu = - R_\mu$ holds. Furthermore, by radiality of $R_\mu$ and Newton's theorem (see, e.\,g., \cite{LL01}), we have the pointwise bound 
$$
\int_{\R^3} \frac{|R_\mu(y)|^2}{|x-y|} \, dy \leq \frac{\| R_\mu \|_{L^2}^2}{|x|}.    
$$
Suppose now that $\| R_\mu \|_{L^2}^2 \leq (2/\pi)$ was true. In view of the simple fact that $\sqrt{-\Delta+\mu^{-2}} \geq \sqrt{-\Delta}$, we  obtain the operator inequality
$$
H \geq \sqrt{-\Delta} - \frac{  \| R_\mu \|_{L^2}^2}{|x|} \geq 0,
$$
by Hardy's inequality $|x|^{-1} \leq \frac{\pi}{2} \sqrt{-\Delta}$, see \cite{H77}. But the nonnegativity of $H$ contradicts the fact that $-1$ is an eigenvalue of $H$. Therefore we deduce that \eqref{ineq:uni3} holds, which completes the proof of Lemma \ref{lem:Rmu_bounds}. \end{proof}

Next, we derive the following strong convergence result for the family $R_\mu$.

\begin{lem} \label{lem:Rmu_CV}
Suppose that $\mu_n \to \infty$ as $n \to \infty$ and let $T_{\mu_n}  \in H^{1/2}(\R^3)$ be a sequence of minimizers as given by Lemma \ref{lem:Tmu}. Furthermore, we denote $R_{\mu_n} = \theta^{1/2}_n T_{\mu_n} \in H^{1/2}(\R^3)$ with $\theta_n = \theta(T_{\mu_n})$ be as above. Then, after passing to a subsequence if necessary, we have that
$$
\mbox{$R_{\mu_n} \to R$ strongly in $L^2(\R^3)$ as $n \to \infty$},
$$ 
where $R \in H^{1/2}(\R^3)$ is a positive radial solution of 
$$
\sqrt{-\Delta} \, R + R - ( |x|^{-1} \ast |R|^2)R = 0,
$$
and it holds that $\| R \|_{L^2}^2 \geq \| Q \|_{L^2}^2$ (i.\,e.~the critical $L^2$-mass for the boson star equation, see \eqref{eq:boson_ground}). 
\end{lem}

\begin{proof}
For notational convenience, we write $R_{n} = R_{\mu_n}$ in what follows.

By the bounds in Lemma \ref{lem:Rmu_bounds}, we can assume that $R_{n} \rightharpoonup R$ weakly in $H^{1/2}$. Moreover, by local Rellich compactness, we can have that $R_{n} \to R$ strongly in $L^2_{\mathrm{loc}}$. We will upgrade this to strong convergence in $L^2$, by deriving a uniform decay estimate for $R_{n}$ as follows. 

We rewrite the equation satisfied by $R_n$ as 
$$
\sqrt{-\Delta + \mu_n^{-2}} \, R_{n} = f_n
$$
with $f_n(x) = (V_n(x) -1) R_{n}(x)$ and $V_n(x) = (|\cdot|^{-1} \ast |R_{n}|^2)(x)$. Since $R_n$ are radial functions and $\| R_n \|_{L^2} \lesssim 1$, we derive from Newton's theorem \cite[Theorem 9.7]{LL01} the uniform pointwise  bound $V_n(x) \leq \| R_n \|_{L^2}^2 |x|^{-1} \lesssim |x|^{-1}$. Moreover, by the fact that $R_n(x) > 0$ is positive, we deduce that
$$
 f_{n}^+(x) := \max \{ 0, f_n(x) \} \equiv 0 \ \ \mbox{for} \ \ |x| \gtrsim 1.
$$
Now, let $G_\mu(x,y)$ be the convolution kernel associated to $(-\Delta+\mu^{-2})^{-1/2}$. From well-known facts (see e.g.\ \cite[p. 183, formula (11)]{LL01}) we have the explicit formula
\begin{align*}
G_\mu(x) & =  \int_0^{\infty} e^{-t \sqrt{-\Delta+\mu^{-2}}}(x,y) \, dt \\
& = \frac{\mu^{-4}}{2\pi^2} \int_0^{\infty} \frac{t}{t^2 + |x-y|^2} K_2 \left (\mu^{-2}  \sqrt{|x-y|^2+t^2} \right ) \, dt ,
\end{align*}
where $K_2$ is the modified Bessel function of the third kind. From the fact that $K_\nu(z) \lesssim |z|^{-\nu}$ for $\mbox{Re} \, \nu > 0$, we easily deduce the uniform bound
$$
0<G_\mu(x,y) \lesssim |x-y|^{-4} .
$$
Hence, by using the posivity of $R_n$ and the fact that $f_n^+$ have compact support in a fixed large ball independent of $n$, we obtain from $R_\mu = (-\Delta+\mu^{-2})^{-1/2} f_n$ that
\begin{align*}
0 < R_n(x) & \leq \int_{\R^3} G_{\mu_n}(x-y) f_n^+(y) \, dy \leq \int_{|y| \lesssim 1} G_{\mu_n} (x-y) f_n^+(y) \, dy  \\
& \leq |x|^{-4} \int_{|y| \lesssim 1} f_n^+(y) \, dy \lesssim |x|^{-4} \ \ \mbox{for} \ \ |x| \gtrsim 1.
\end{align*}
In the last step, we used that
$$
\int_{|y| \lesssim 1} f_{n}^+(y) \, dy \lesssim \int_{|y| \lesssim 1 } \frac{1}{|y|} R_n(y) \, dy \lesssim 1,
$$
thanks to Newton's theorem and $\| R_n \|_{L^2} \lesssim 1$ again.

In summary, we have shown that
\begin{equation*}
R_n(x) \lesssim |x|^{-4} \quad \mbox{for $|x| \gtrsim 1$ and $n \geq 1$}.
\end{equation*}
Using this decay bound (which is square integrable at infinity), we easily see that $R_n \to R$ stronlgy in $L^2_{\mathrm{loc}}$ implies that 
\begin{equation*}
\mbox{$R_n \to R$ stronlgy in $L^2(\R^3)$ as $n \to \infty$}.
\end{equation*}
In view of the lower bound in Lemma \ref{lem:Rmu_bounds}, we deduce that $R \not \equiv 0$ holds. By passing to the limit in the equation satisfied by $R_n$, we deduce that $R$ satisfies the equation displayed in Lemma \ref{lem:Rmu_CV}.

Finally, integrating the equation satisfied by $R$ against \[\Lambda R = x \cdot \nabla R + \frac{3}{2} R=\frac{\mathrm{d}}{\mathrm{d}a} a^{\frac32}R(ax)\Big|_{a=1}\] we obtain that 
$$
\int_{\R^3} R \sqrt{-\Delta} \, R \, dx = \frac{1}{2} \int_{\R^3}  ( |x|^{-1} \ast |R|^2 ) |R|^2 \, dx.
$$
Recall the interpolation estimate
$$
\int_{\R^3} ( |x|^{-1} \ast |f|^2 ) |f|^2 \,dx \leq C_{\mathrm{opt}} \left ( \int_{\R^3} \overline{f} \sqrt{-\Delta} f \right ) \left ( \int_{\R^3} |f|^2 \right )
$$
for all $f \in H^{1/2}$, where the optimal constant is given by $C_{\mathrm{opt}} = 2/ \| Q \|_{L^2}^2$; see, e.\,g., \cite[Appendix]{L07} for this fact. Hence we deduce that $\| R \|_{L^2}^2 \geq \| Q \|_{L^2}^2$. This completes the proof of Lemma \ref{lem:Rmu_CV}. \end{proof}

Having the result of Lemma \ref{lem:Rmu_CV} at hand, we can now prove the following illposedness following in the case of positive mass parameter $m>0$.

\begin{pro} \label{pro:uniform2}
Suppose that $m>0$ holds in \eqref{eq:boson-star}, and let $t>0$. Then the map $\phi \mapsto u(t)$ fails to be uniformly continuous for initial data in the set \[M = \{ \phi \in L^2 : \phi \text{ radial }, \| \phi \|_{L^2}^2 \geq K \}\] with respect to the $L^2$-norm, where $K \geq \| Q \|_{L^2}^2$ is some universal constant.
\end{pro}

\begin{proof}
Recall that we can assume $m=1$ without loss of generality. For $\mu_1, \mu_ 2 > 0$ given, we consider the solitary wave solutions
$$
u_{\mu_1}(t,x) = e^{it \mu_1}Q_{\mu_1}(t,x), \quad u_{\mu_2}(t,x) = e^{it \mu_2} Q_{\mu_2}(t,x),
$$
where $Q_\mu(x)$ are radial positive solutions to \eqref{eq:Qmu} that are given by $Q_\mu=\mu^{3/2} R_\mu(\mu x)$ with $R_\mu$ taken from Lemma \ref{lem:Tmu}. Following the proof of Proposition \ref{pro:uniform}, we define 
$$
I_{\mu_1, \mu_2}(t) = \| u_{\mu_1}(t) - u_{\mu_2}(t) \|_{L^2}^2.
$$
Similarly as above, we find that
\begin{align*}
I_{\mu_1, \mu_2}(t) & =  \|Q_{\mu_1} \|_{L^2}^2 + \| Q_{\mu_2} \|_{L^2}^2 \\
& \quad - 2 \cos (t (\mu_1-\mu_2)) \left( \frac{\mu_1}{\mu_2} \right )^{3/2} \int_{\R^3} R_{\mu_1}(x) R_{\mu_2} \left ( \left ( \frac{\mu_1}{\mu_2} x \right ) \right ) \, dx
\end{align*}  
Now let $t> 0$ be given. Define the sequence $\mu(n) = \frac{\pi}{2t} n^2$ for $n \in \N$, which ensure that $\cos(t (\mu(n+1) - \mu(n)) = 0$ for all $n \in \N$. By Lemma \ref{lem:Rmu_CV} and after possibly passing to a subsequence, we have that $R_{\mu(n)} \to R$ strongly in $L^2$ with a radial positive function $R \not \equiv 0$.   Hence, we conclude
$$
\lim_{n \to \infty} I_{\mu(n+1), \mu(n)} (t) = 2 \| R \|_{L^2}^2 \neq 0 .
$$
On the other hand, we have
\begin{align*}
&\lim_{n \to \infty} I_{\mu(n+1), \mu(n)} (0) \\
={}& 2 \| R \|_{L^2}^2- 2 \lim_{n \to \infty} \Big( \frac{\mu(n+1)}{\mu(n)} \Big)^{\frac32} \int_{\R^3} R_{\mu(n)}(x) R_{\mu(n+1)} \Big(\frac{\mu(n+1)}{\mu(n)} x   \Big)\, dx \\
={}& 2 \| R \|_{L^2}^2 - 2 \int_{\R^3} |R|^2 \, dx = 0.
\end{align*}
This completes the proof of Proposition \ref{pro:uniform2}. \end{proof}

\bibliographystyle{plain} \bibliography{refs-hartree}
\end{document}